\documentclass[11 pt]{amsart}
\usepackage{amscd,amsfonts,amssymb,amsmath}
\usepackage{hyperref}

\usepackage{tikz}
\usepackage[margin=3.7 cm]{geometry}

\newtheorem{theorem}{Theorem}[section]
\newtheorem{cor}[theorem]{Corollary}

\newtheorem{prop}[theorem]{Proposition}

\theoremstyle{definition}

\newtheorem{example}[theorem]{Example}

\numberwithin{equation}{subsection}
\theoremstyle{plain}

\newtheorem{question}{Question}
\newtheorem{problem}{Problem}
\newtheorem{corollary}[theorem]{Corollary}

\numberwithin{equation}{section}

\begin{document}
\title{Linearity problem for non-abelian tensor products}
\author[V. G. Bardakov]{Valeriy G. Bardakov}
\author[A. V. Lavrenov]{Andrei V. Lavrenov}
\author[M. V. Neshchadim]{Mikhail V. Neshchadim}

\date{\today}
\address{Sobolev Institute of Mathematics, Novosibirsk 630090, Russia,}
\address{Novosibirsk State University, Novosibirsk 630090, Russia,}
\address{Novosibirsk State Agrarian University, Dobrolyubova street, 160,  Novosibirsk, 630039, Russia,}
\email{bardakov@math.nsc.ru}

\address{Saint-Petersburg State University, University Embankment, 7--9, Saint-Petersburg, 199034, Russia,}
\email{A.V.Lavrenov@gmail.com}

\address{Sobolev Institute of Mathematics and Novosibirsk State University, Novosibirsk 630090, Russia,}
\email{neshch@math.nsc.ru}

\subjclass[2010]{Primary 20E25; Secondary 20G20, 20E05}
\keywords{non-abelian tensor product; linear group; faithful linear representation}

\begin{abstract}
In this paper we give an example of a linear group such that its tensor square is not linear. Also, we formulate some sufficient conditions for the linearity
of non-abelian tensor products $G \otimes H$ and tensor squares $G \otimes G$. Using these results we prove that tensor squares of some groups with one relation and some knot groups are linear. We prove that the Peiffer square of a finitely generated linear group is linear. At the end we construct faithful linear representations for the non-abelian tensor square of a free group and free nilpotent group.
\end{abstract}
\maketitle

\section{Introduction}

R. Brown and J.-L. Loday \cite{BL, BL1}
introduced the non-abelian tensor product $G \otimes H$ for a pair of groups $G$ and $H$
 following  works of Miller \cite{Mil},  and Lue \cite{Lue}.
They showed that the third homotopy group of the suspension of
an Eilenberg-MacLane space $K(G, 1)$ satisfies
$$
\pi_3 S K(G, 1) \cong J_2 (G),
$$
where $J_2 (G)$ is the kernel of the derived map $\kappa : G \otimes G \longrightarrow G'$, $g \otimes h
\mapsto [g, h] = g^{-1} h^{-1} g h$. Hence there exists the short exact sequence
$$
0 \longrightarrow \pi_3 S K(G, 1) \longrightarrow G \otimes G  \longrightarrow G' \longrightarrow 1.
$$

Also, the non-abelian tensor product is used to describe the
third relative homotopy group of a triad as a non-abelian tensor product of the
second homotopy groups of appropriate subspaces. More specifically, let a $CW$-complex $X$ be the union $X = A \cup B$ of two path-connected $CW$-subspaces $A$ and $B$ whose intersection $C = A \cap B$ is path-connected. If the canonical
homomorphisms $\pi_1 (C) \longrightarrow \pi_1 (A)$, $\pi_1 (C) \longrightarrow \pi_1 (B)$ are surjective, then, according
to \cite{BL},
$$
\pi_3 (X, A, B) \cong \pi_2 (A, C) \otimes \pi_2 (B, C),
$$
where the groups $\pi_2 (A, C)$ and $\pi_2 (B, C)$ act on one another via $\pi_1 (C)$.

The investigation of the non-abelian tensor product from a group theoretical
point of view started with a paper by Brown, Johnson, and Robertson \cite{BJR}. They compute the non-abelian tensor square of all non-abelian
groups of order up to 30 using Tietze transformations.

One of the topics of research on the non-abelian tensor products has been to
determine which group properties are preserved by non-abelian tensor products. By
using homological arguments, Ellis \cite{E} showed that if $G$ and $H$ are finite groups,
then $G \otimes H$ is also finite.
Visscher \cite{V} proved that if $G$, $H$ are solvable
(nilpotent), then $G \otimes H$ is solvable (nilpotent) and gives a bound on the nilpotency class of $G \otimes H$.
In \cite{DLT} it was proved that the  tensor product of groups
of nilpotency class at most $n$ is a group of nilpotency class at most $n$, thereby
improving the bound given by Visscher. For other results in this direction see the survey of
Nakaoka \cite{N}.

In this paper we study the linearity problem for non-abelian tensor products.
Let $n$ be a positive integer and let $P$ be a field. A group $G$ is said to be {\it linear of degree $n$ over $P$} if it is isomorphic with a subgroup of $GL_n(P)$, the group of all $n \times n$ non-singular matrices over $P$ or, equivalently, if it is isomorphic with a group of invertible linear transformations of a vector space of dimension $n$ over $P$ (see \cite{Mal}).
We study the following

\begin{question}
Let $G$ and $H$ be linear groups. Are the groups $G \otimes H$, $G \otimes G$ linear?
\end{question}

We show that in  general  the answer is negative. More accurately, we prove that the tensor square  $SL_n(\mathbb{Q}) \otimes SL_n(\mathbb{Q})$ of the special linear group
$SL_n(\mathbb{Q})$ over the field of rational numbers is not linear for $n \geq 3$. On the other side we formulate some sufficient conditions under which  the groups $G \otimes H$,  $G \otimes G$ are linear. Using these conditions, we prove that the non-abelian tensor squares of some groups with one defining relation and groups of fibered   knots are linear. If $G$ is a finitely generated free group or finitely generated free nilpotent group, then we construct concrete faithful linear representations for $G \otimes G$.

The non-abelian tensor square $G \otimes G$ is connected to other group constructions: exterior tensor square  $G \wedge G$ and Peiffert square $G \bowtie G$. We prove that if $G$ is finitely generated, then $G \bowtie G$ is linear.

We note that the following problems are still open.

\begin{problem}
1) Let $G$ be a finitely generated linear group. Is the group $G \otimes G$ linear?

2) Let $G$ be a  linear group. Is the group $G \bowtie G$ linear?
\end{problem}

\medskip

\textbf{Acknowledgement.} The authors are grateful to  V.~Thomas, S.~O.~Ivanov,  V.~Petrushenko, A.~Yu.~Ol'shanskii, V.~A.~Roman'kov, M.~Singh  for  useful discussions. Also we thank to J.~O.~Button and O.~V.~Bryukhanov for examples of linear groups with non-linear abelization (see section 5). The present work was supported by  Russian Science Foundation (project No. 16-41-02006).

\section{Preliminaries}\label{section2}

In this paper  we shall use the following notations.  For elements $x$, $y$ in a group
$G$, the conjugation  of $x$ by $y$ is $x^y = y^{-1} x y$; and the commutator of $x$ and $y$ is
$[x, y] = x^{-1} x^y = x^{-1} y^{-1} x y$. We write $G'$ for the derived subgroup of $G$, $G^{ab}$ for the
abelianized group $G / G'$.

Recall the definition of the non-abelian tensor product $G \otimes H$ of groups $G$ and $H$ (see \cite{BL, BL1}).
This tensor product is defined for any pair of groups $G$ and $H$ where each one acts on the other (on the
right)
$$
G \times H \longrightarrow G, ~~(g, h) \mapsto g^h; ~~~H \times G \longrightarrow H, ~~(h, g) \mapsto h^g
$$
and on itself by conjugation, in such a way that for all $g, g_1 \in G$ and $h, h_1 \in H,$
$$
g^{(h^{g_1})} = \left( \left( g^{g_1^{-1}} \right)^h \right)^{g_1}~~
\mbox{and}~~
h^{(g^{h_1})} = \left( \left( h^{h_1^{-1}} \right)^g \right)^{h_1}.
$$
In this situation we say that $G$ and $H$ act {\it compatibly} on each other. The {\it non-abelian
tensor product} $G \otimes H$ is the group generated by all symbols $g \otimes h$, $g \in G$, $h \in H$,
subject to the relations
$$
g g_1 \otimes h = (g^{g_1} \otimes h^{g_1}) (g_1 \otimes h)~ \mbox{and}~~ g \otimes h h_1 = (g \otimes h_1) (g^{h_1} \otimes h^{h_1})
$$
for all $g, g_1 \in G$, $h, h_1 \in H$.

In particular, as the conjugation action of a group $G$ on itself is compatible, then
the tensor square $G \otimes G$ of a group $G$ may always be defined. Also, the tensor product $G \otimes H$ is defined if $G$ and $H$ are two normal subgroups of some group $M$ and actions are conjugations in $M$.

\medskip

Recall { \it the main diagram for the non-abelian tensor square } (see \cite{BL, BL1}).
Let $G$ be a group. One of the main tools for studying of the non-abelian tensor square $G \otimes G$ is the following diagram:

$$
\begin{array}{ccccccccc}
                 &             &                &                              & 0         &             & 0          &             &  \\
                 &             &                &                              &\downarrow &             & \downarrow &             &  \\
  H_3(G)         & \rightarrow & \Gamma(G^{ab}) & \stackrel{\psi}{\rightarrow} & J_2(G)    & \rightarrow & H_2(G)     & \rightarrow & 0\\
  \shortparallel &             & \shortparallel &                              &\downarrow &             & \downarrow &  &  \\
  H_3(G)         & \rightarrow & \Gamma(G^{ab}) & \stackrel{\psi}{\rightarrow} & G\otimes G& \rightarrow &G\wedge G& \rightarrow & 1\\
                 &             &                &                              &\downarrow &             & \downarrow &             &  \\
                 &             &                &                              & G'        &     =       & G'         &             &  \\
                 &             &                &                              &\downarrow &             & \downarrow &             &  \\
                 &             &                &                              & 1         &             & 1          &             &  \\
\end{array}
$$
with exact rows and columns.
Here

1) $H_2(G)$, $H_3(G)$ are the second and the third homology groups of $G$ with the coefficients  in the trivial $\mathbb{Z}G$-module
 $\mathbb{Z}$. The second homology group  $H_2(G)$
for the group $G=F/R$, where $F$ is a free group, can be find by the Hopf formula:
$$
H_2(G)\cong (F' \cap R)/[F,R].
$$

2) $G\wedge G$ is the exterior product of $G$ onto itself.
For the group $G=F/R$ it can be presented in the form (see \cite{BFM})
$$
G\wedge G \cong F'/[F,R].
$$
In particular, if $G$ is a free group, then $G\wedge G \cong G'$.

3) $J_2(G) = \pi_3 SK(G, 1)$ is the kernel of the derived map $\kappa : G \otimes G \longrightarrow G'$, which on the generators of
$G \otimes G$ is defined by the rule:
$$
g_1 \otimes g_2 \mapsto [g_1,g_2].
$$
The group $J_2(G)$  lies in the center $Z(G \otimes G)$
and its elements are invariant under the action of $G$ onto $G \otimes G$,
which is defined by the formula
$$
(g_1 \otimes g_2)^{g} =g_1^{g} \otimes g_2^{g}.
$$
 In particular, if $g_2 = g_1$, then
$$
(g_1 \otimes g_1)^{g} =g_1 \otimes g_1
$$
for any $g, g_1 \in G$.

4) $\Gamma(G^{ab})$ is Whitehead's quadratic functor.
The group $\Gamma(G^{ab})$ is generated by elements $\gamma(g G')$
and $\psi$ is defined by the formula
$$
\gamma(g G') \longmapsto g \otimes g.
$$

The image $\psi \Gamma(G^{ab})$ is not equal in the general case to the group
$J_2(G)$ since $J_2(G)/ \psi \Gamma(G^{ab}) \cong H_2(G)$;

For the functor  $\Gamma : A \longmapsto \Gamma (A)$, where $A$ is an abelian group it is known that

a) $\Gamma (A\times B)\cong \Gamma (A)\times\Gamma (B)\times (A\otimes_{\mathbb{Z}} B)$,
where $A\otimes_{\mathbb{Z}} B$ is the abelian tensor product of abelian groups;

b)
$\Gamma (\mathbb{Z}_n)\cong
\displaystyle\left\{
\begin{array}{ll}
  \mathbb{Z}_n    & n \equiv 1~ (\mbox{mod}~2),\\
  \mathbb{Z}_{2n} & n \equiv 0~ (\mbox{mod}~2).\\
\end{array}
\right.
$

c) $\Gamma (\mathbb{Z})\cong  \mathbb{Z}$.

In particular, $\Gamma (\mathbb{Z}^n)\cong  \mathbb{Z}^{\frac{n(n+1)}{2}}$.

\section{Linearity problem}

In this section we will use a result of Malcev \cite{Mal} (see also \cite[Chapter 2]{Weh}) on the linearity of abelian groups. To formulate it recall some definitions.
If $G$ is any group $\tau(G)$ is the subgroup of $G$ generated by all the
periodic normal subgroups of $G$; that is $\tau(G)$ is the maximum periodic
normal subgroup of $G$. $G$ has {\it finite rank at most} $n$ if every finite subset
of $G$ is contained in an $n$-generator subgroup of $G$. If $G$ is abelian and
periodic then $G$ has finite rank at most $n$ if and only if for each prime $p$
the Sylow $p$-subgroup of $G$ is a direct product of at most $n$ cyclic and
Pr\"{u}fer $p$-groups (a Pr\"{u}fer $p$-group is a $\mathbb{C}_{p^{\infty}}$-group). If $\pi$ is any set of
primes and $G$ is a group with a unique maximal $\pi$-subgroup we denote
this maximal $\pi$-subgroup by $G_{\pi}$.

Malcev proved:

i) An abelian group $A$ has a faithful representation of degree $n \geq 1$ over some field of characteristic zero if and only if
$\tau(G)$ has rank at most $n$.

ii) An abelian group $A$ has a faithful representation of degree $n \geq 1$ over
some field of characteristic $p > 0$ if and only if $\tau(G)_{p'}$ (here $p'$ denotes all primes except $p$) has finite rank $r$ and
$\tau (A)_p$ has finite exponent $p^e$ satisfying
$$
p^{e - 1} + \max \{ 1, r \} < n + 1.
$$

\medskip
We are ready to prove the following

\begin{prop}
There is a linear group $G$ such that $G \otimes G = G \wedge G$ is not linear.
\end{prop}

\begin{proof}
For a perfect group $G=G'$ it follows from the main diagram (see section \ref{section2}) that $G\otimes G=G\wedge G$ and the sequence
$$
0 \longrightarrow  H_2(G,\,\mathbb Z) \longrightarrow G\otimes G \longrightarrow G \longrightarrow 0
$$
is exact.

For $n\geq 3$ the group $SL_n (\mathbb{Q})$ is perfect and its second homology group coincides with the $\mathrm K_2$-group of the field $\mathbb Q$,
$$
 H_2(SL_n (\mathbb{Q}),\,\mathbb Z)=\mathrm K_2(\mathbb Q),
$$
see~\cite[Corollary~11.2]{Milnor}.

Next,
$$
\mathrm K_2(\mathbb Q)=\{\pm1\}\times\prod\limits_{p\text{ odd prime}}(\mathbb Z/p)^\times
$$
by~\cite[Theorem~11.6]{Milnor}, so that  $\mathrm K_2(\mathbb Q)$ contains an abelian 2-group of infinite rank and unbounded exponent. Using Maltsev's criterion we conclude that such a group can not be linear. Therefore the group
$$
SL_n (\mathbb{Q}) \otimes SL_n (\mathbb{Q}) = SL_n (\mathbb{Q}) \wedge SL_n (\mathbb{Q})
$$
is not linear as well.
\end{proof}

\medskip

To study the linearity problem for the non-abelian tensor product we can use a presentation of
a tensor product as a central extension (see, for example, \cite{DLT}.) The {\it derivative subgroup} of $G$ by $H$ is defined to be the following subgroup
$$
D_H(G) = [G, H] = \langle g^{-1} g^h ~|~g\in G, h \in H \rangle.
$$
The map $\kappa : G \otimes H \longrightarrow D_H(G)$ defined by $\kappa (g \otimes h) = g^{-1} g^h$ is a homomorphism, its kernel is the central subgroup of
$G \otimes H$  and $G$ acts on $G \otimes H$ by the rule $(g \otimes h)^x = g^x \otimes h^x$, $x \in G$. There exists the short exact sequence
$$
1 \longrightarrow A \longrightarrow G \otimes H  \longrightarrow D_H (G) \longrightarrow 1.
$$
In this case $A$ can be viewed as a $\mathbb{Z} [D_H (G)]$-module via conjugation in $G \otimes H$, i. e. under the action induced by setting
$$
a \cdot g = x^{-1} a x,~~a \in A, x \in G \otimes H, \kappa(x) = g.
$$

We can formulate some sufficient conditions for the linearity of $G \otimes H$.
It is well known that the tensor product $G \otimes H$ with trivial actions is isomorphic to the abelian tensor product
 $G^{ab}\otimes_{\mathbb{Z}} H^{ab}$. Hence, in this case  the question on the linearity of $G \otimes H$ is equivalent to the question on the linearity of the abelian tensor product and the answer follows from the Malcev theorem.

Further we will assume that the action of $G$ on $H$ or the action of $H$ on $G$ is non-trivial.
We have  the following short exact sequence
$$
0 \longrightarrow A \longrightarrow G\otimes H \longrightarrow D_H(G) \longrightarrow 1.
\eqno{(1)}
$$
Note that  $A$ is the kernel of the natural map  $G\otimes H \longrightarrow D_H(G)$,
$g\otimes h \longrightarrow g^{-1} g^h$, $g \in G$, $h \in H$, and is a central subgroup of
 $G\otimes H$.

\begin{prop}
Let the following conditions hold

1) $A$, $D_H(G)$ are linear groups;

2) $H^2(D_H(G), A) = 0$, in particular, this condition holds if $A$ is divisible or $D_H(G)$ is a  free group.

Then  $G\otimes H = A \times D_H(G)$ is the direct product and is a linear group.

\end{prop}

\begin{proof}
It is well known that if $H^2(D_H(G), A) = 0$, then the sequence
$$
0 \longrightarrow A \longrightarrow G \otimes H \longrightarrow D_H(G) \longrightarrow 1,
$$
splits. In particular, this condition holds if $A$ is divisible or $D_H(G)$ is a free group.

Since $A$ is a central subgroup, then
$G\otimes H\cong A \times D_H(G)$ and is a linear group as a direct product of linear groups.

\end{proof}

The main problem in the use of this theorem is the description of the central subgroup $A$. For the tensor square we can use another approach.

Let us formulate some sufficient conditions under which
$G \otimes G$ is a direct product of the commutator subgroup  $G'$
and the Whitehead group  $\Gamma(G^{ab})$.

\begin{theorem}\label{t2}
Let  $H_2(G)=H_3(G)= H_2(G') = 0$ and  one of the following conditions hold

1) $H^2(G', \Gamma(G^{ab})) = 0$;

2) $\Gamma(G^{ab})$ is a divisible group;

3) $G' / G''$ is a free abelian group.

Then
$$
G \otimes G \cong   \Gamma(G^{ab}) \times G'.
$$
If moreover $G$ is finitely generated and $G'$ is linear, then $G \otimes G$ is linear.
\end{theorem}

\begin{proof}
Since, $H_2(G)=H_3(G)=0$, then the main diagram has the form

$$
\begin{array}{ccccccccc}
    &             &                &                              & 0         &             & 0          &             &  \\
    &             &                &                              &\downarrow &             & \downarrow &             &  \\
  0 & \rightarrow & \Gamma(G^{ab}) & \stackrel{\psi}{\rightarrow} & J_2(G)    & \rightarrow & 0          & \rightarrow & 0\\
    &             & \shortparallel &                              &\downarrow &             & \downarrow &  &  \\
  0 & \rightarrow & \Gamma(G^{ab}) & \stackrel{\psi}{\rightarrow} & G\otimes G& \rightarrow &G\wedge G& \rightarrow & 1\\
    &             &                &                              &\downarrow &             & \downarrow &             &  \\
    &             &                &                              & G'        &     =       & G'         &             &  \\
    &             &                &                              &\downarrow &             & \downarrow &             &  \\
    &             &                &                              & 1         &             & 1          &             &  \\
\end{array}
$$

From this diagram  $J_2(G)= \Gamma(G^{ab})$ and $G\wedge G =G'$.
Hence we have  the short exact sequence
$$
0 \longrightarrow \Gamma(G^{ab}) \longrightarrow G\otimes G \longrightarrow G' \longrightarrow 1.
$$
If $H^2(G', \Gamma(G^{ab})) = 0$, then this sequence is  splittable:
$$
G\otimes G= \Gamma(G^{ab}) \leftthreetimes G'.
$$
As we know if $\Gamma(G^{ab})$ is divisible or $G'$ is free, then $H^2(G', \Gamma(G^{ab})) = 0$. Let us show that if $G' / G''$ does not have torsion, then $H^2(G', \Gamma(G^{ab})) = 0$.
Indeed, by the universal coefficient theorem there is the following short exact sequence
$$
0 \longrightarrow Ext_{\mathbb{Z}} (H_{1} (G'), \Gamma(G^{ab})) \longrightarrow H^2 (G', \Gamma(G^{ab})) \longrightarrow \mathrm{Hom}_{\mathbb{Z}} (H_{2} (G'), \Gamma(G^{ab})) \longrightarrow 0.
$$
Since  $H_2(G') = 0$  we have the short exact sequence
$$
0 \longrightarrow Ext_{\mathbb{Z}} (H_{1} (G'), \Gamma(G^{ab})) \longrightarrow H^2 (G', \Gamma(G^{ab})) \longrightarrow 0.
$$
Hence $H^2 (G', \Gamma(G^{ab})) = 0$ if and only if $Ext_{\mathbb{Z}} (H_{1} (G'), \Gamma(G^{ab})) = 0$. It is known that if $H_1(G')$ is free abelian, then
$Ext_{\mathbb{Z}} (H_{1} (G'), \Gamma(G^{ab})) = 0$.

Since $\Gamma(G^{ab})$ is a central subgroup, this product is the direct product:
$$
G\otimes G= \Gamma(G^{ab}) \times G'.
$$

If $G$ is a finitely generated, then $G^{ab}$ is finitely generated abelian group and $\Gamma(G^{ab})$ is also a finitely generated abelian group. Then $G \otimes G$ is linear as a direct product of two linear groups.
\end{proof}

As  consequence we get the following result

\begin{cor} \cite{BFM}
Let $F_n$ be a free group of rank  $n$.
Then
$$
F_n\otimes F_n \cong \mathbb{Z}^{n(n+1)/2} \times (F_n)'.
$$
\end{cor}

\section{Groups with one defining relation and knot groups}

Let $G$ be a group with one defining relation:
$$
G=\left\langle \, X \, \| \, r=1 \, \right\rangle,
$$
where $r \not\in F'$, $F=\left\langle \, X \,  \right\rangle$.
Then   $H_k(G)=0$, $k \geq 2$ (see \cite[p. 49]{Br}).
Hence, there exists the following short exact sequence:
$$
0 \longrightarrow \Gamma(G^{ab}) \longrightarrow G\otimes G \longrightarrow G' \longrightarrow 1.
$$
If $G^{ab}$ does not have torsion, then   $G^{ab}$ is a free abelian group
and  $\Gamma(G^{ab})$ is a free abelian group. Then, if $H^2(G') = 0$, then $H^2(G', \Gamma(G^{ab})) = 0$, which follows from the decomposition
$$
H^k(S, A \oplus B) = H^k(S, A) \oplus H^k(S, B)
$$
for every group $S$ and all $S$-modules $A$ and $B$.

From Theorem \ref{t2} follows

\begin{prop}
Let $G$ be a group with one defining relation:
$$
G=\left\langle \, X \, \| \, r=1 \, \right\rangle,
$$
where $r \not\in F'$, $F=\left\langle \, X \,  \right\rangle$ such that $H^2(G') = 0$. If one from the following conditions hold:

1) $G^{ab}$ does not have torsion;

2) $G' / G''$ is a free abelian group.

Then $G \otimes G = \Gamma (G^{ab}) \times G'$. If moreover $G$ is finitely generated and $G$ is linear, then $G \otimes G$ is linear.
\end{prop}

\medskip

 It is well known that if $K$ is a tame knot in 3-sphere $\mathbb{S}^3$ and $G_K = \pi_1 (\mathbb{S}^3 \setminus K)$ its group, then $H_n (G_K) = 0$ for $n >1$ (see, for example \cite[p. 5]{Kaw}).  Recall that a knot $K$  is called {\it fibered}  if there is a 1-parameter family $F_t$ of Seifert surfaces for $K$, where the parameter $t$ runs through the points of the unit circle $S^{1}$, such that if $s$ is not equal to $t$ then the intersection of $F_{s}$ and $F_t$ is exactly $K$.
The commutator subgroup $G_K'$ for the fibered knot $K$ is a free group of finite rank \cite{CZ} and $G_K$ is linear \cite{Ash}.

\begin{prop}
Let $K$ be a tame fibered knot in 3-sphere $\mathbb{S}^3$, then $G_K \otimes G_K = G_K' \times \mathbb{Z}$ and has a faithful linear representation into $GL_2(\mathbb{Z} [t, t^{-1}])$.
\end{prop}

\begin{proof}
It is well known that $G_K^{ab} = \mathbb{Z}$ and then $\Gamma (G_K^{ab})  = \mathbb{Z}$.
From Theorem \ref{t2} it follows that $G_K \otimes G_K = \mathbb{Z}  \times G_K'$.

To construct a linear representation, use the fact that $G_K'$ is a free group of finite rank and by Sanov's theorem \cite[Chapter 5]{KM} it has a faithful linear representation into $SL_2(\mathbb{Z}) \leq GL_2(\mathbb{Z} [t, t^{-1}])$. Define a linear representation of $G_K^{ab} = \mathbb{Z} = \langle \gamma \rangle$ into $GL_2(\mathbb{Z} [t, t^{-1}])$ by the rule
$$
\gamma \longmapsto \left(
                     \begin{array}{cc}
                       t & 0 \\
                       0 & t \\
                     \end{array}
                   \right).
$$
Since the image of $\gamma$ is a scalar matrix, i.e. lies in the center of $GL_2(\mathbb{Z} [t, t^{-1}])$, we constructed a faithful linear representation of
$G_K \otimes G_K$.
\end{proof}

\begin{example}
1) The braid group $B_3$ on 3 strings has presentation
$$
B_3 = \langle \sigma_1, \sigma_2 ~||~\sigma_1 \sigma_2 \sigma_1 = \sigma_2 \sigma_1 \sigma_2 \rangle.
$$
and is the group of trefoil knot. The commutator subgroup $B_3'$ is a free group of rank 2. Hence  the tensor square $B_3 \otimes B_3 = \mathbb{Z} \times F_2$ has a faithful linear representation into $GL_2(\mathbb{Z} [t, t^{-1}])$.

2) It is known that the group of the figure eight knot has a presentation
$$
G = \langle x, y ~||~y x^{-1} y x y^{-1} = x^{-1} y x y^{-1} x \rangle
$$
and is a fibered knot. Hence  the tensor square $G \otimes G = \mathbb{Z} \times G'$ has a faithful linear representation into $GL_2(\mathbb{Z} [t, t^{-1}])$.
\end{example}

In the first example we shown  that $B_3 \otimes B_3 = \mathbb{Z} \times F_2$. On the other side $B_3$ contains the pure braid group $P_3$, which is normal in $B_3$, has index 6 and is the direct product of the center, which is isomorphic to $\mathbb{Z}$, and a free group of rank 2. Hence, $B_3 \otimes B_3$ is isomorphic to $P_3$ and we proved

\begin{prop}
There is a non-trivial non-abelian group $G$ such that the tensor square $G \otimes G$ is isomorphic to a proper subgroup of $G$.
\end{prop}

\begin{question}
1) Is it true that $B_n \otimes B_n$, $n > 3$, is linear?

2) Is it true that for arbitrary tame knot $K$ the group $G(K) \otimes G(K)$ is linear?

\end{question}

\section{On the linearity of the Peiffer product} Recall the definition of the
Peiffer product. Given $G$ and $H$ acting compatibly on each other, in \cite{W} the {\it Peiffer product} $G \bowtie H$  was defined their as the quotient of the free product $G * H$ by the normal closure $K$ of all elements of the form
$$
h^{-1} g^{-1} h g^h ~\mbox{or}~g^{-1} h^{-1} g h^g
$$
where $g \in G$ and $h \in H$.
Whitehead \cite{W} posed a question on the asphericity of subcomplexes of aspherical 2-complexes and reformulated it as part of the wider problem of finding conditions under which the groups $G$ and $H$ are embedded in $G \bowtie H$.

In \cite{GH} it was proved that if  $\varphi : G*H \to G \bowtie H$, then modulo $K = Ker (\varphi)$, $h g \equiv g h^g$, so that every element of $G \bowtie H$ can be written as $\varphi(g) \varphi(h)$ for suitable $g$, $h$. Denote $\varphi(g) \varphi(h)$ as $\langle g, h \rangle$. The relations
$$
\langle g, h \rangle \langle g_1, h_1 \rangle = \langle g g_1, h^{g_1} h_1 \rangle = \langle g g_1^{h^{-1}}, h h_1 \rangle
$$
are defining relations for $G \bowtie H$ on the generators $\langle g, h \rangle$ and so $G \bowtie H$ is a homomorphic image  of both the semidirect products
$G \ltimes H$ and $G \rtimes  H$. The group $G \bowtie H$ is obtained from $G \ltimes H$ (or from $G \rtimes  H$) by imposing the relations
$$
(g^{-1} g^h, 1) = (1, h^{-g} h).
$$

If $G$ and $H$ act on one another trivially, then $G \bowtie H$ is just the direct product $G \times H$ and $K = G \square H$, where $G \square H$ is the Cartesian subgroup of $G * H$ (the kernel of the
canonical homomorphism $G * H \longrightarrow G \times H$.

From \cite[Proposition 2.1]{GH} follows
$$
G \bowtie G \cong G^{ab} \times G^{ab}.
$$
Using this isomorphism one can prove

 \begin{prop}
Let $G$ be a linear group and $G$ is finitely generated or $G = G'$, then $G \bowtie G$ is linear.
 \end{prop}

From this proposition it follows that if  $G = SL_n(\mathbb{Q})$, $n \geq 2$, then $G \bowtie G$ is linear. On the other side, we know that
$SL_n(\mathbb{Q}) \otimes SL_n(\mathbb{Q})$ and $SL_n(\mathbb{Q}) \wedge SL_n(\mathbb{Q})$ are not linear for $n \geq 3$.

Note that this proposition is not true for  arbitrary linear group $G$ since there are linear groups with nonlinear abelization.

\begin{example}
1) (O. V. Bryukhanov) Let 
$G = \overset{\infty}{\underset{i=2}{\ast}} \mathbb{Z}_i$
 be the free product of cyclic groups. Then $G$ is linear as the free product of linear groups. On the other side, by Malcev criteria (see Section 3) the abelization $G^{ab}$ is not linear.
 
2) (J. O. Button) 
Take the set of matrices 
$$
A_i = \left(  \begin{array}{cc}
                       1 & x^i \\
                       0 & 1 \\
                     \end{array}
                   \right) \in SL_2(\mathbb{Z}[x]),~~~i \in \mathbb{N}. 
$$
Then $A = \langle A_i~|~i \in \mathbb{N} \rangle$ is a free abelian group of countable rank.
Put
$$
B = \left(
                     \begin{array}{cc}
                       3 & 0 \\
                       0 & 1 \\
                     \end{array}
                   \right) \in GL_2(\mathbb{Q}). 
$$
It is easily to check that these matrices satisfy the relations
$$
B A_i B^{-1} = A_i^3, ~~~i \in \mathbb{N}.
$$
Hence the group generated by $A_i$ and $B$ has the presentation
$$
G_2 = \langle A_i,  i \in \mathbb{N}, B ~||~[A_i, A_j] = 1, B A_i B^{-1} = A_i^3, ~~~i, j \in \mathbb{N} \rangle,
$$
which is a subgroup of $GL_2 (\mathbb{Q}[x])$, but its abelization $G_2^{ab} \cong \bigoplus\limits_{i=1}^{\infty} \mathbb{Z}_2 \oplus \mathbb{Z}$ does not have a faithful linear representations over field of  characteristic $p \not= 2$.

Analogically, take the set of matrices
$$
C_i = \left(  \begin{array}{cc}
                       1 & y^i \\
                       0 & 1 \\
                     \end{array}
                   \right) \in SL_2(\mathbb{Z}[y]),~~~i \in \mathbb{N}.
$$
Then $C = \langle C_i~|~i \in \mathbb{N} \rangle$ is a free abelian group of countable rank.
Put
$$
D = \left(
                     \begin{array}{cc}
                       4 & 0 \\
                       0 & 1 \\
                     \end{array}
                   \right) \in GL_2(\mathbb{Q}).
$$
It is easily to check that these matrices satisfy the relations
$$
D C_i D^{-1} = C_i^4, ~~~i \in \mathbb{N}.
$$
Hence the group generated by $C_i$ and $D$ has the presentation
$$
G_3 = \langle C_i,  i \in \mathbb{N}, D ~||~[C_i, C_j] = 1, D C_i D^{-1} = C_i^4, ~~~i, j \in \mathbb{N} \rangle,
$$
which is a subgroup of $GL_2 (\mathbb{Q}[y])$, but its abelization $G_3^{ab} \cong \bigoplus\limits_{i=1}^{\infty} \mathbb{Z}_3 \oplus \mathbb{Z}$ does not have a faithful linear representations over field of  characteristic $p \not= 3$. 

Let us take $G = G_2 \oplus G_3$. Then it is metabelian and has a faithful linear representation in $GL_4 (\mathbb{Q}[x, y])$, but its abelization $G^{ab} \cong \bigoplus\limits_{i=1}^{\infty} \left( \mathbb{Z}_2 \oplus \mathbb{Z}_3 \right) \oplus \mathbb{Z} \oplus \mathbb{Z}$ is not linear.
\end{example}

It is evident that
the following short exact sequence holds
 $$
1 \longrightarrow 1 \times G' \longrightarrow G^{ab} \times G \longrightarrow G^{ab} \times G^{ab} \longrightarrow 1.
 $$
Since  $G^{ab} \times G \cong G \bowtie G$ we can add in the main diagram new terms.

\begin{prop}
The following diagram holds
$$
\begin{array}{ccccccccc}
       &             &                &                              & 0         &             & 0          &             &  \\
       &             &                &                              &\downarrow &             & \downarrow &             &  \\
H_3(G) & \rightarrow & \Gamma(G^{ab}) & \stackrel{\psi}{\rightarrow} & J_2(G)    & \rightarrow & H_2(G)     & \rightarrow & 1\\
\shortparallel &     & \shortparallel &                              &\downarrow &             & \downarrow &  &  \\
H_3(G) & \rightarrow & \Gamma(G^{ab}) & \stackrel{\psi}{\rightarrow} & G\otimes G& \rightarrow &G\wedge G& \rightarrow & 1\\
       &             &                &                              &\downarrow &             & \downarrow &             &  \\
       &             &                &                              &G \bowtie G &     =       &G \bowtie G &             &  \\
       &             &                &                              &\downarrow &             & \downarrow &             &  \\
       &             &                &                              &G^{ab} \times G^{ab} & = &G^{ab} \times G^{ab} & &= H_1(G) \times H_1(G)\\
       &             &                &                              &\downarrow &             & \downarrow &             &  \\
       &             &                &                              & 1         &             & 1          &             &  \\
\end{array}
$$
\end{prop}

\section{Faithful linear representations}

In the paper \cite{BFM} it was proved:

1) If $F_n$ is the  free group of rank  $n$, then
$$
F_n\otimes F_n \cong \mathbb{Z}^{n(n+1)/2} \times (F_n)'.
$$

2) If $N_{n,c}=F_n/\gamma_c F_n$ is the free nilpotent group of rank $n > 1$ and class  $c \geq 1$,
then
$$
N_{n,c}\otimes N_{n,c} \cong \mathbb{Z}^{n(n+1)/2} \times (N_{n,c+1})'.
$$

\begin{prop}
Let $G$ be a free countable group. Then the exterior square $G \wedge G$ has a faithful representation into  $SL_2(\mathbb{Z})$ and the tensor square $G \otimes G$ has a faithful representation into  $GL_2(\mathbb{C})$.
\end{prop}

\begin{proof}
As was proven in \cite{BL1}, for the free group  $G$ there are  isomorphisms
$$
G \wedge G \cong   G',~~~G \otimes G \cong   \Gamma(G^{ab}) \times G'.
$$
Since $G$ is free, its commutator subgroup   $G'$ is free.
Hence, by the Sanov result \cite[Chapter 5]{KM} there is a faithful representation of $G'$ into $SL_2(\mathbb{Z})$ and the first part of the proposition holds.

Further,  $\Gamma(G^{ab})$ is a free abelian group.
Let $a_k$, $k\in I$ be its free generators.
Take transcendental elements  $t_k$, $k\in I$
in the field  $\mathbb{C}$, which are algebraically independent over $\mathbb{Q}$.
Then the matrix group
$$
T= \left\langle \,
\left(%
\begin{array}{cc}
  t_k & 0 \\
  0   & t_k \\
\end{array}%
\right) \,\, \| \,\, k\in I \,
   \right\rangle
$$
is isomorphic to the group  $\Gamma(G^{ab})$. If $\varphi : G' \longrightarrow GL_2(\mathbb{Z})$
is an embedding, then
$$
\left\langle \,
\varphi G',\, T  \,
   \right\rangle \cong
\varphi G' \times  T.
$$
Hence, the group  $G \otimes G$ has a faithful representation in the matrix group over the ring
$\mathbb{Z}[t_k^{\pm 1}, k\in I]$.

If $G = F_{\infty}$ is countably generated then it has a faithful representation into $SL_2(\mathbb{Z})$. To prove that $\Gamma (F_{\infty}^{ab})$ is linear we use the following property
$$
\Gamma (F_{\infty}^{ab}) = \Gamma(\lim F_n^{ab}) = \lim (\Gamma F_n^{ab}).
$$
\end{proof}

For the finitely generated free groups from this theorem follows

\begin{corollary}
The tensor square  $F_n\otimes F_n$ has a faithful representation into
$GL_2(\mathbb{Z}[t_1^{\pm 1},\ldots, t_m^{\pm 1}])$, where $m=\frac{n(n+1)}{2}$.
\end{corollary}

For the free nilpotent groups we can prove

\begin{prop}
There is a faithful representation
$$
N_{n,c} \otimes N_{n,c} \longrightarrow T_{c+2}(\mathbb{C})
$$
into the group of triangular matrices $T_{c+2}(\mathbb{C})$.
\end{prop}

\begin{proof}
We noted that
$$
N_{n,c}\otimes N_{n,c} \cong \mathbb{Z}^{n(n+1)/2} \times (N_{n,c+1})'.
$$
Hence, we have to define faithful linear representations for $\mathbb{Z}^{n(n+1)/2} = \langle a_1, a_2, \ldots, a_m \rangle$, $m = n(n+1)/2$ and for $(N_{n,c+1})'$, where $N_{n,c+1} = \langle x_1, x_2, \ldots, x_n \rangle$. Let
$$
\tau_1, \tau_2, \ldots, \tau_m, t_{ij}, i = 1, 2, \ldots, n, j = 1, 2, \ldots, c+1,
$$
be  complex numbers which are algebraically independent over $\mathbb{Q}$. Define the following maps
$$
a_k \mapsto \tau_k E \in T_{c+2}(\mathbb{C}), ~~k = 1, 2, \ldots, m,
$$
which  defines a faithful representation of $\mathbb{Z}^{n(n+1)/2}$ into $T_{c+2}(\mathbb{C})$,
and
$$
x_i \mapsto \left(
              \begin{array}{ccccccc}
                1 & t_{i1} & 0 & 0 & \ldots & 0 & 0 \\
                0 & 1 & t_{i2} & 0 & \ldots & 0 & 0 \\
                \ldots & \ldots & \ldots & \ldots & \ldots & \ldots & \ldots \\
                0 & 0 & 0 & 0 & \ldots & 1 & t_{ic} \\
                0 & 0 & 0 & 0 & \ldots & 0 & 1 \\
              \end{array}
            \right) \in UT_{c+2} (\mathbb{C})
 ~~i = 1, 2, \ldots, n.
$$
As Romanovskii proved \cite{Rom} the map, defined on $x_i$ is a faithful representation of $N_{n,c+1}$ into $UT_{c+2} (\mathbb{C})$.  Hence we have a faithful representation of $\mathbb{Z}^{n(n+1)/2} \times N_{n,c+1}$ into $T_{c+2} (\mathbb{C})$. Since $\mathbb{Z}^{n(n+1)/2} \times (N_{n,c+1})'$ is a subgroup of
$\mathbb{Z}^{n(n+1)/2} \times N_{n,c+1}$, we have the needed representation.
\end{proof}

\end{document}